   \documentclass[11pt]{amsart}

   \usepackage{color, amsmath,amssymb, amsfonts, amstext,amsthm, latexsym}

   \numberwithin{equation}{section}
   \newtheorem{lemma}{Lemma}[section]
   \newtheorem{theorem}[lemma]{Theorem}
   \newtheorem{remark}[lemma]{Remark}

   \newtheorem{definition}[lemma]{Definition}

   \newcommand{\EX}{{\Bbb{E}}}
   \newcommand{\PX}{{\Bbb{P}}}

\newcommand{\cF}{{\cal F}}

\newcommand{\PP}{{\mathbb P}}
\newcommand{\s}{\sigma}
\newcommand{\e}{\varepsilon}

\newcommand{\om}{\omega}
\newcommand{\Om}{\Omega}
\newcommand{\D}{\Delta}
\newcommand{\de}{\delta}
\newcommand{\p}{\partial}
\renewcommand{\k}{\kappa}

\renewcommand{\cF}{\mathcal F}

\title
[Stochastic Two-Layer  Geophysical Flows] {Small probability events
for two-layer geophysical flows under uncertainty
}

\author{Aijun Du}
   \address[Aijun Du]
   {Department of Applied Mathematics\\
   Illinois Institute of Technology\\
   Chicago, IL 60616, USA}

  \author{Jinqiao Duan}
   \address[Jinqiao Duan]
   {Department of Applied Mathematics\\
   Illinois Institute of Technology\\
   Chicago, IL 60616, USA}
   \email
   [Jinqiao Duan]{duan@iit.edu}

 \author{Hongjun Gao}
   \address[Hongjun Gao]
   {Department of Mathematics\\
   Nanjing Normal University\\
   Nanjing, China}
   \email
   [Hongjun Gao]{gaohj@njnu.edu.cn}

\date{October 15, 2008}

\subjclass[2000]{Primary 60H15, 76U05; Secondary  86A05, 34D35}
\keywords{Stochastic partial differential equations, stochastic
geophysical flow models, large deviations, random dynamical
systems }

\begin{document}
\vsize=8.6in \hsize=5.6in

\begin{abstract}

The stochastics two-layer quasi-geostrophic  flow model is an
intermediate system between the single-layer two dimensional
barotropic flow model and the continuously stratified  three
dimensional  baroclinic flow model. This model is widely used to
investigate basic mechanisms in geophysical flows, such as
baroclinic effects, the Gulf Stream   and subtropical gyres.

A large deviation principle for the  two-layer quasi-geostrophic
flow model under uncertainty is proved. The proof is based on  the
Laplace principle and  a variational approach. This approach does
not require the exponential tightness estimates which are needed
in other methods for establishing large deviation principles.

\end{abstract}

 \maketitle


\section{Introduction}\label{s1}

The continuously stratified, three dimensional (3D) baroclinic
quasi-geostrophic flow model describes large scale geophysical
fluid motions in the atmosphere and oceans. This model is much
simpler than the primitive flow model or the rotating
Navier-Stokes flow model. When the fluid density  is approximately
constant, this model reduces to the barotropic, single-layer,  two
dimensional (2D) quasi-geostrophic model. The two-layer
quasi-geostrophic  flow model,  in which the fluid consists of two
homogeneous fluid layers of  uniform but distinct densities
$\rho_1$ and $\rho_2$, is an intermediate system between the
single-layer 2D barotropic flow model and the continuously
stratified, 3D baroclinic flow model.
\\

The two-layer quasi-geostrophic flow model  has been used as a
theoretical and   numerical model to understand basic mechanisms
in large scale geophysical flows, such as baroclinic effects
\cite{Ped87}, wind-driven circulation \cite{Berloff, Berloff2},
the Gulf Stream \cite{Huang}, fluid stability \cite{Benilov} and
subtropical gyres \cite{Oz}.
 Recently Salmon \cite{Salmon} introduced a generalized
two-layer ocean  flow model.
\\

We prove a large deviation  principle for this stochastic infinite
dimensional system, by a recent weak convergence approach, based
on a variational representation for functionals of infinite
dimensional Brownian motion \cite{BD00, BD07}. In this approach,
the large deviations for SPDEs  are derived by
  showing some qualitative properties (well-posedness, compactness
  and weak convergence) of certain perturbations of
  the original SPDEs. This  method has been
  recently applied in  several papers on
  SPDEs \cite{Sundar, WangDuan, DuanMillet} or SDEs
  in infinite dimensions \cite{Ren}.\\

More information about this weak convergence approach for large deviations
   in the finite dimensional setting can be found in the book \cite{Dupuis}.
It is different from
  other existing approaches,
  which usually require extra exponential tightness estimates,
  for establishing large deviation principles for
  SPDEs
  \cite{CardonWeber, Cerrai-Rockner, Chang, ChenalMillet,  Chow, FW92, Kallian-Xiong, Peszat, Sowers, Zab}.
An alternative approach \cite{Feng} for large deviations is based
on nonlinear semi-group theory and infinite dimensional
Hamilton-Jacobi equations; it also requires to
establish exponential tightness.\\

This paper is organized as follows. The mathematical formulation
for the stochastic two-layer  geophysical flow model is in the
next section. Then the well-posedness for the model is discussed
in \S \ref{s3}. Finally, a large deviation principle is shown in
\S \ref{s4}.

\section{Mathematical setup} \label{s2}

We consider the two-layer quasi-geostrophic flow model
(\cite{Ped87}, p. 423; \cite{Salmon}, p.87):

\begin{equation}\label{q1}
\begin{split}
\frac{\partial q_1}{\partial t} + J(\psi_1, q_1 + \beta  y )
        &  = \nu \D^2  \psi_1  + f   +  \sqrt{\epsilon}\sigma_1(q_1,q_2)\dot{W^1},\\
\frac{\partial q_2}{\partial t} + J(\psi_2, q_2 + \beta  y )
       &  = \nu \D^2  \psi_2  -r \D  \psi_2 +  \sqrt{\epsilon}\sigma_2(q_1,q_2)\dot{W^2},
\end{split}
\end{equation}
with boundary condition
\begin{equation*}
\psi_1 =  \psi_2 = 0   \;,\\ \;q_1 = q_2 = 0,
\end{equation*}

where potential vorticities $q_1(x, y,t) $, $q_2(x, y,t)$ for
 the
top layer and the bottom layer are
defined via stream functions $\psi_1(x, y,t)$, $\psi_2(x, y,t)$, respectively,
\begin{equation}\label{pv1}
\begin{split}
 q_1 &=  \D \psi_1 - F_1  \cdot (\psi_1 - \psi_2),
\\
 q_2 &=  \D \psi_2 - F_2  \cdot (\psi_2 - \psi_1).
\end{split}
\end{equation}

\begin{remark}
The boundary conditions $\  \psi_1 =  \psi_2 = 0   \; , q_1 = q_2
= 0 \; $ give: $\   \D \psi_1  =  \D \psi_2 = 0$ on the boundary.
\end{remark}

Here  $x,\,y $ are  Cartesian coordinates in zonal (east), meridional (north) directions,  respectively; $(x,y)\in
D:=(0,L)\times (0,L)$, where $L$ is a positive number;
 $F_1, F_2$ are positive constants defined by
\begin{equation*}
\begin{split}
F_1 = \frac{f_0^2 }{gh_1}\frac{\rho_0}{\rho_2-\rho_1},\\
F_2 = \frac{f_0^2 }{gh_2}\frac{\rho_0}{\rho_2-\rho_1},
\end{split}
\end{equation*}
with $g$ the gravitational acceleration;
$h_1, h_2$ the depth of top and bottom layers,
  $\rho_1, \rho_2$ the densities  ($\rho_2> \rho_1$)
of top and bottom layers, respectively;  and
$L, \rho_0$ the characteristic scales for horizontal
length and density of the flows,
 respectively;
 $f_0 +\beta y$ (with $f_0, \beta$ constants) is the Coriolis
parameter and $\beta$ is the meridional gradient of the
Coriolis  parameter;   $ \nu > 0$ is  the viscosity.
Note that  $r = f_0 \frac{\delta_E}{2(h_1+h_2)}$ is the  Ekman
constant  which
 measures the intensity of   friction at the bottom boundary layer
(the so-called Ekman layer) or the rate for
vorticity decay due to the friction in the Ekman layer. Here
$\delta_E= \sqrt{2\nu/f_0}$ is the Ekman layer thickness
(\cite{Ped87}, p.188).
   Moreover,  $J(h, g)
=h_xg_y -h_yg_x$ is the Jacobi operator and $\D=\p_{xx} +\p_{yy}$ is the Laplace operator.

Now, we set $W =\left( \begin{array}{cc}
W^1 \\
W^2\end{array} \right)\;$ and $\dot{W} =\left( \begin{array}{cc}
\dot{W^1}\\
\dot{W^2} \end{array} \right)$. An important part of the above
flow model (\ref{q1})is the white noise term $\dot{W}$, which is
the generalized  derivative of a Wiener process $W(t)$ with
respect to time $t$, in an appropriate function space to be
specified below. This white noise term  $\left( \begin{array}{cc}
\sqrt \e \s_1\dot{W^1}\\
\sqrt \e \s_2\dot{W^2} \end{array} \right)\;$, with $\e>0$ a small
parameter and $\s_i$ noise intensity, describes the fluctuating
part of the external wind forcing in both of the fluid layers; see
Arnold \cite{Arn00}. The fluctuating part  is usually of a shorter
time scale than the response time scale of the large scale
quasi-geostrophic flows. So we neglect the autocorrelation time of
this fluctuating process. We thus assume the noise is white in
time but it is allowed to be colored in space, i.e., it may be
correlated in space variables $x$ and $y$. The Wiener process
(also a Gaussian process) $W(t)$ has zero mean and is
characterized by its covariance operator $Q$. There has been some
analysis on wind stress curl data from the National Aeronautics
and Space Administration Scatterometer (NSCAT) and from the
National Center for Environmental Prediction (NCEP); see, for
example, \cite{Milliff, Chin}.  Such data analysis also involves
estimating the covariance and its trace, and the trace is usually
taken to be finite. In this paper, we consider the case when
 the covariance operator $Q$ of  the
Wiener process has a finite trace.\\

In the following, $L^2(D)$, $V=H^1_0(D)$  denote the standard
scalar Lebesgue and Sobolev spaces. Let $\mathcal{D}(A) =
\mathbb{H}$$^2(D) \bigcap \mathbb{H}$$^1_0(D)$, where
$\mathbb{H}$$ = L^2(D) \times L^2(D)$  and $\mathbb{H}$$^1_0(D)=
H^1_0(D) \times H^1_0(D)$ are product vector spaces. The scalar
product and the induced norm in $ L^2(D)$ or $\mathbb{H}$
are denoted as $(\cdot, \cdot)$ and $\|\cdot\|$, respectively. \\

$W(t)$ is a Wiener processes  defined  on a filtered probability
space $(\Om, \cF, \cF_t, \PX)$, taking values in $\mathbb{H}$. We
denote $Q$ as its associated covariance operator, it is a linear
symmetric positive covariant operator in the Hilbert space
$\mathbb{H}$. We assume that $Q$ is trace class, i.e., $tr(Q) <
\infty$.

As in \cite{PZ92} and \cite{Sundar}, let ${\mathbb{H}}_0 = Q^{\frac12}{\mathbb{H}}$. Then ${\mathbb{H}}_0$ is a Hilbert space with the scalar product
$$
(q, \psi)_0 = (Q^{-\frac12}q, Q^{-\frac12}\psi),\; \forall q, \psi \in {\mathbb{H}}_0,
$$
together with the induced norm $|\cdot|_0=\sqrt{(\cdot,
\cdot)_0}$\;. The embedding $i:{\mathbb{H}}_0 \to  {\mathbb{H}}$
is Hilbert-Schmidt and hence compact, and moreover, $i \; i^* =Q$.

Let $L_Q$ be the space of linear operators $S$ such that
$SQ^{\frac12}$ is a Hilbert-Schmidt operator (and thus a compact
operator) from ${\mathbb{H}}$ to ${\mathbb{H}}$. The norm in the
space $L_Q$ is $\|S\|_{L_Q} =tr (SQS^*)$, where $S^*$ is the
adjoint operator of $S$.

 With these notations, the above two-layer system can be rewritten as:

\begin{equation}
\begin{split}
\frac{\partial q}{\partial t} &= [Aq + F(q)]  +  \sqrt{\epsilon}\sigma(q)\dot{W},\\
q(0) &= \xi,
\end{split}
\end{equation}
or
\begin{equation}\label{q2}
\begin{split}
dq &= [Aq + F(q)]dt  +  \sqrt{\epsilon}\sigma(q)dW ,\\
q(0) &= \xi,
\end{split}
\end{equation}

 where
\begin{equation*}
q = \left( \begin{array}{cc}
q_1  \\
q_2  \end{array} \right), \\
\;\sigma(q)\dot{W} = \left( \begin{array}{cc}
\sigma_1(q_1,q_2)\dot{W^1}  \\
\sigma_2(q_1,q_2)\dot{W^2}  \end{array} \right),
\end{equation*}

\begin{equation*}
Aq =\left( \begin{array}{cc}
-\nu \D^2  \psi_1 \\
-\nu \D^2  \psi_2   \end{array} \right),
\end{equation*}

\begin{equation*}
F(q) =\left( \begin{array}{cc}
-J(\psi_1, q_1 + \beta  y   ) +  f \\
-J(\psi_2, q_2 + \beta  y )    -r \D  \psi_2 \end{array} \right).
\end{equation*}

Here $q_1,q_2$ and $\psi_1,\psi_2$ are defined in \ref{pv1}.

In order to obtain the weak solution, the noise intensity $\s:
{\mathbb{H}} \to L_Q({\mathbb{H}}_0, {\mathbb{H}})$ is assumed to
satisfy the following conditions.

 \medskip

\textbf{Assumption A:}\\
\\
A.1.  The noise intensity $\sigma(\cdot) : {\mathbb{H}} \rightarrow L_Q({\mathbb{H}}_0; {\mathbb{H}})$ is continuous.\\
\\
A.2. $\|\s(q)\|^2_{L_Q({\mathbb{H}}_0;  \mathbb{H})} \leq K
(1+\|q\|_\mathbb{H}^2),$
for some positive constant $K$. \\
\\
A.3. There exists a constant L such that for all $q, \psi \in
{\mathbb{H}}$, we have
$\|\s(q)-\s(\hat{q})\|^2_{L_Q({\mathbb{H}}_0;{\mathbb{H}})} \leq L
\|q -\hat{q} \|_{\mathbb{H}}^2$.\\

For strong solution, the noise intensity $\s$ is assumed to
satisfy the following additional conditions.

 \medskip

\textbf{Assumption A':}\\
\\
A'.1.  The noise intensity $\sigma(\cdot) : {\mathbb{H}} \rightarrow L_Q({\mathbb{H}}_0; {\mathbb{H}_0^1})$ is continuous.\\
\\
A'.2. $\|\s(q)\|^2_{L_Q({\mathbb{H}}_0;  \mathbb{H}_0^1)} \leq K'
(1+\|q\|_{\mathbb{H}_0^1}^2),$
for some positive constant $K'$. \\


\bigskip


\section{Well-posedness }\label{s3}

\subsection{Well-posedness of two-layer system}

To treat the nonlinearity in the two-layer fluid model we need the
following lemmas:

\begin{lemma} \label{Jacobi}
The Jacobian operator has the following properties:
\begin{align}
&J(u,v) = - J(v,u),\qquad (J(u,v),v)  = 0,\\
&(J(u,v),w) = (J(v,w),u),
\end{align}
for $u,v,w$ in $H^1$. Moreover the following estimates hold:
\begin{equation}\label{jac-3}
|( J(u,v),\D u)|\le c_0\|\D v\|\cdot\|\nabla u\|\cdot \|\D
u\|,\quad u,\,v\in H^2;
\end{equation}
\begin{equation*}
|( J(u,v),w)|\le c_1\|\D u\|\cdot \|\D v\|\cdot \|w\|,\quad
u,\,v\in H^2,\,w\in L^2;
\end{equation*}
\begin{equation}\label{jac-2a}
|(J(u,v),w)|\le c_1\|\nabla u\|\cdot \|\D v\|\cdot\|\nabla
w\|,\quad u,\,w\in H^1\,, v\in H^2.
\end{equation}
\end{lemma}

The detailed proof of the above lemma can be found in
\cite{DuanFursikov}.

\begin{lemma}\label{bound}
Let $F_1$ and $ F_2$ be positive constants, and assume that
$q_i,\psi_i , (i=1,2)$ satisfy
\begin{equation*}
\begin{split}
 q_1 &=  \D \psi_1 - F_1  \cdot (\psi_1 - \psi_2),
\\
 q_2 &=  \D \psi_2 - F_2  \cdot (\psi_2 - \psi_1).\\
\end{split}
\end{equation*}
Then we have
\begin{equation*}
\|\psi_1\|^2_{H^2} + \|\psi_2\|^2_{H^2} \leq C (\|q_1\|^2+ \|q_2\|^2).
\end{equation*}
\end{lemma}

\begin{proof}
On any finite interval $t \in [0,T]$, we take inner product on
both sides of the equations with $\triangle \psi_1$ and $\triangle
\psi_2$ respectively :
\begin{equation*}
\begin{split}
\int_D q_1 \D\psi_1 dx &=  \int_D |\D \psi_1|^2 dx + F_1 \int_D
|\nabla \psi_1|^2dx - F_1 \int_D \nabla\psi_2 \nabla \psi_1dx,
\\
\int_D q_2 \D\psi_2 dx &=  \int_D |\D \psi_2|^2 dx + F_2 \int_D
|\nabla \psi_2|^2dx - F_2 \int_D \nabla\psi_2 \nabla \psi_1dx.\\
\end{split}
\end{equation*}

Multiply both sides of first equation by $F_2$, second equation by
$F_1$ , then add them together :
\begin{eqnarray*}
& & F_2\int_D | \D\psi_1|^2 dx + F_1 \int_D |\D\psi_2|^2 dx\\
&+& F_1F_2\int_D (|\nabla \psi_2|^2 + |\nabla \psi_1|^2 -
2\nabla\psi_2
\nabla \psi_1) dx \\
&=& F_2 \int_D q_1 \D\psi_1dx + F_1 \int_D \D \psi_2 q_2 dx.
\end{eqnarray*}

Noticing that $F_1F_2\int_D (|\nabla \psi_2|^2 + |\nabla \psi_1|^2
- 2\nabla\psi_2 \nabla \psi_1) dx$ is always non-negative, we
have:
\begin{eqnarray*}
& & F_2\int_D | \D\psi_1|^2 dx + F_1 \int_D |\D\psi_2|^2 dx\\
&\leq& F_2 \int_D q_1 \D\psi_1dx + F_1 \int_D \D \psi_2 q_2 dx\\
&\leq& F_2 \int_D |q_1| |\D\psi_1| dx + F_1 \int_D |\D \psi_2| |q_2| dx\\
\end{eqnarray*}

By the Young's Inequality:
\begin{eqnarray*}
& &F_2\int_D | \D\psi_1|^2 dx + F_1 \int_D |\D\psi_2|^2 dx \\
&\leq& {1 \over 2}F_2 \int_D |q_1|^2dx + {1 \over 2}F_2 \int_D
|\D\psi_1|^2dx + {1 \over 2} F_1 \int_D |q_2|^2 dx + {1 \over
2}F_1 \int_D |\D\psi_2|^2dx.
\end{eqnarray*}
Let $C={max\{F_1,F_2\} \over min\{F_1,F_2\}}$, we get :
\begin{eqnarray*}
\int_D | \D\psi_1|^2 dx + \int_D |\D\psi_2|^2 dx \leq C( \int_D
|q_1|^2dx +  \int_D |q_2|^2 dx),
\end{eqnarray*} which is simply
\begin{eqnarray}
\|\D\psi_1\|_{L^2}^2  + \|\D\psi_2\|_{L^2}^2  \leq C(
\|q_1\|_{L^2}^2 +  \|q_2\|_{L^2}^2),
\end{eqnarray}

Since $D=(0, L) \times (0, L)$,  by the Poincare's inequality and
the inequality (6.4) in \cite{lad}, we know that $\|\psi_i
\|_{H^2}$ can be upper bounded by $\|\D\psi_i\|_{L^2}$ multiplying
by a positive constant. Thus there exists a positive constant $C$,
such that
\begin{eqnarray}
\|\psi_1\|_{H^2}^2  + \|\psi_2\|_{H^2}^2  \leq C( \|q_1\|_{L^2}^2
+  \|q_2\|_{L^2}^2).
\end{eqnarray}
\end{proof}

In the proof of well-posedness of the system, we need the
following random version of the Gronwall's inequality, which is
Lemma 3.9 in \cite{DuanMillet} with minor modification.

\begin{lemma} (\cite{DuanMillet}) \label{gronwall}
Let $X, Y$ and $I$ be non decreasing, non-negative processes,
$\varphi$ be a non-negative processes and $Z$ be a non-negative
integrable random variable. Assume that $\int_0^T\varphi(r)dr \le
M$ almost surely and there exist positive constants $\alpha$ and
$\beta \le \frac{1}{2(1+ ME^M)}$,  $\tilde C > 0$ and $\bar C >
0$(depending on $M$) such that
$$ X(t) + \alpha Y(t) \le Z + \int_0^t\varphi(r) X(r) dr + I(t),
a.s.
$$
$$\EX (I(t)) \le \beta\EX(X(t)) + C\int_0^t\EX(X(r))dr + \tilde C.
$$
Then if $X\in L^{\infty}([0, T]\times\Omega)$, we have for $t\in [0,
T]$,
$$\EX X(t) + \alpha\EX Y(t) \le \bar C(1 + \EX (Z)).$$
 \end{lemma}

\subsection{Well-posedness of perturbed two-layer system}

The solution for the stochastic two-layer  geophysical flow problem  under random influences  is denoted as $q^\e$, although often we omit the
$\e$ here. The goal for this paper is to show the large deviation principle (or equivalently, the Laplace principle) for the
family $q^\e$.

Let $\mathcal{A}$ be the  class of ${\mathbb{H}}_0-$valued $(\cF_t)-$predictable stochastic processes $q$ with the property $\int_0^T
\|q(s)\|^2 ds < \infty, \; $ a.s. Let
\[S_M=\Big\{h \in \mathcal{A}\ \mbox{and}\ h \in L^2(0, T; H_0): \int_0^T \|h(s)\|^2 ds \leq M\Big\}.\]
The set $S_M$ endowed with the following weak topology is a
  Polish space (complete separable metric space)
\cite{BD07}:
$ d_1(h, k)=\sum_{i=1}^{\infty} \frac1{2^i} \big|\int_0^T
(h(s)-k(s), \tilde{e}_i(s))_0 ds \big|,$
where $ \{\tilde{e}_i(s)\}_{i=1}^{\infty}$ is a complete
orthonormal basis for $L^2(0, T; {\mathbb{H}}_0)$. Define
\begin{equation} \label{AM}
 \mathcal{A}_M= \{ q \in \mathcal{A}: q(\om) \in S_M, a.s.  \}.
\end{equation}
 As in \cite{Sundar}, we prove existence and uniqueness of the solution to the stochastic two-layer  geophysical flow
equation. However,  in the sequel, we will need some precise
bounds on the norm of the solution to a more general equation,
which contains  an extra forcing term driven by an element of
 $ \mathcal{A}_M$. More precisely, let $h\in \mathcal{A}$ and
 consider the following
generalized two-layer system with initial condition $q_h(0)=\xi$,
\begin{equation} \label{Benardgene}
d q_h(t) + \big[ Aq_h(t) +F(q_h(t))\big]dt =
\s(q_h(t)) dW(t) + \tilde{\s}(q_h(t)) h(t) dt. 
\end{equation}

\begin{remark}
The noise intensity $\tilde{\s}: {\mathbb{H}} \to L_Q({\mathbb
{H}}_0, {\mathbb{H}})$ is assumed to satisfy the same conditions as
$\s$(Assumption A).
\end{remark}

\begin{definition} (Weak solution)\\
Recall that a stochastic process $q_h(t,\om)$ is called the weak
solution for the generalized stochastic two-layer
quasi-geostrophic flow problem  \eqref{Benardgene}  on $[0, T]$
with initial condition $\xi$ if $q_h$  is in $C([0, T];
\mathbb{H}) \cap L^2((0, T); \mathbb{H}_0^1)$, a.s., and satisfies
\begin{eqnarray}\label{weak}
 (q_h(t), \psi)-(\xi, \psi) + \int_0^t[ (q_h(s), A\psi)
 +\big( F(q_h(s)), \psi\big) ]ds \nonumber \\
 =  \int_0^t (\s(q_h(s)) dW(s),\psi)
 + \int_0^t \big(\tilde{\s}(q_h(s)) h(s),\psi\big) ds,\;\; a.s.,
\end{eqnarray}
for all $\psi \in \mathcal{D}(A)$ and all $t \in [0,T]$.
\end{definition}
 In the following,  we work in the Banach space
$ X: = C\big([0, T]; \mathbb{H}\big) \cap L^2\big((0, T);
\mathbb{H}_0^1\big) $
with the norm
\begin{eqnarray}\label{norm}
\|q\|_X = \Big\{\sup_{0\leq s \leq T}\|q(s)\|^2+ \int_0^T \sup_{0\le
\tau \le s}\|\nabla q(\tau)\|^2 ds\Big\}^\frac12.
\end{eqnarray}

\begin{definition} (Mild solution and strong solution)\\
Given $\xi \in \mathbb{H}$, mild solution to the stochastic
evolutionary equation \eqref{Benardgene} on the stochastic space
$(\Om, \cF, \cF_t, \PX)$ is a $\cF_t$-adapted process $q_h(t)\in
C([0,T];$$\mathbb{H}$$) \bigcap L^2(0,T;$$\mathbb{H}$$^1_0)$,
satisfying the following integral form:
\begin{equation*}
q_h(t) = S(t) \xi + \int^t_0 S(t-s)[\tilde{\s}(q_h(s)) \,h(s)-
F(q_h(s))] ds + \int^t_0 S(t-s)\sigma(q_h(s))dW(s),
\end{equation*}
where $S(t)$ is the semigroup generated by the linear (unbounded)
operator $A$.  Moreover, $q_h$ is   called a strong solution if $u
\in C([0,T];\mathbb{H}$$^1_0) \bigcap L^2(0,T;\mathcal{D}(A))$ for
  $q_h(0)=\xi \in \mathbb{H}$$^1_0$.
\end{definition}

\begin{theorem}\label{wellposeness} (Well-posedness )\\
If Assumptions (A1, A.2, A.3) hold  and  the initial datum satisfies
that $\EX \|\xi\|^2 < \infty$, then the mild solution to equation
\eqref{Benardgene} is unique. Moreover, we fix $M>0$, then for any
$h\in \mathcal{A}_M$, there exists a pathwise unique weak solution
$q_h$ of the generalized stochastic two-layer quasi-geostrophic flow
problem \eqref{Benardgene} with initial condition  $q_h(0)=\xi
\in{\mathbb{H}}$  such that $q_h\in X$ a.s. There also exists a
constant $C_1=C_1(T,M,K, F_1, F_2 , \nu ,r)$ such that
\begin{equation} \label{boundgeneral}
\EX\Big( \sup_{0\leq s \leq T} \|q_h(s)\|^2 + \int_0^T \sup_{0\leq
\tau \leq s}\|\nabla q_h(\tau)\|^2\, ds \Big) \leq C_1\, \big(
1+\EX\|\xi\|^2\big).
\end{equation}
Furthermore, if the initial condition $q_h(0) = \xi \in$
$\mathbb{H}$$^1_0$ and the additional Assumptions (A'1, A'.2)
hold, then the solution is strong and unique.
\end{theorem}

\begin{remark}
   Note that when $h\equiv 0$ and $\s$ is multiplied
by $\sqrt{\e}$ with any positive constant $\e$, we deduce that the stochastic two-layer  geophysical flow  equation  has a unique
weak solution. Note that here $\e$ does not have to be small. On the other hand, if the covariance operator $Q\equiv 0$, i.e.,
when the Gaussian process $W$ vanishes, we also deduce the existence and uniqueness of  the solution to the deterministic
 control equation defined in terms of an element $h\in  L^2((0,T);{\mathbb{H}})$ and an initial condition $\xi \in {\mathbb{H}}$
\begin{equation}\label{dc}
d q(t) + \big[ Aq(t) +F(q(t))\big]dt =\s(q(t)) h(t) dt, \;\;
q(0)=\xi.
\end{equation}
 If  $h \in S_M$,  the solution $q$ to \eqref{dc} satisfies
\begin{equation} \label{control-norm2}
  \sup_{0\leq s \leq T} \|q(s)\|^2+\int_0^T \sup_{0\leq \tau \leq s}\|\nabla q(\tau)\|^2 ds
  \leq C_2(F_1, F_2 , \nu ,r, K, T, M, \EX\|\xi\|).
\end{equation}
\end{remark}



\begin{proof}  We   refer to \cite{PZ92}
and \cite{DuanMillet} for the existence of the mild solution. Here
we only give a priori estimates   for the solutions of
\eqref{Benardgene} to guarantee the existence of the strong
solutions. For   simplicity, we suppress the $h-$dependence in
$q$'s and
we also omit the subscript $L^2$ in various norms below.\\
We define $\|q\|^2_{L^2}=\|q_1\|^2_{L^2} + \|q_2\|^2_{L^2}$. The Ito
formula for $\|q_1\|^2_{L^2}$ and using Assumption A  gives:
\begin{eqnarray*}
\|q_1\|^2 &=& \|q_1(0)\|^2 +2 \nu  \int_0^t
(\triangle^2 \psi_1, q_1)ds\\
& &- 2\int^t_0(J(\psi_1,q_1+\beta y),q_1)ds + \int^t_0(f,q_1)ds +
\int^t_0 (\tilde{\s}(q_1,q_2)h_1,q_1)ds\\
& &+ 2\int_0^t\|\s_1(q_1,q_2)\|^2_{L_Q(H_0; H)}ds + 2\int^t_0
({\s}_1(q_1,q_2)dW(s),q_1)ds\\
&\leq&\|q_1(0)\|^2 - \nu  \int_0^t \|\nabla q_1\|^2ds
+\int^t_0\|f\|^2ds\\
&+&C\int^t_0 (1+ \|h_1\|)(\|q_1\|^2 + \|q_2\|^2)ds + |M_1(\tau)|,
\end{eqnarray*}
where
\begin{eqnarray*}
M_1(\tau) = 2\int^\tau_0 ({\s}_1(q_1,q_2)dW(s),q_1)ds.
\end{eqnarray*}
Therefore
\begin{eqnarray} \label{q1-up}
& &  \|q_1\|^2 +  \nu  \int_0^t  \|\nabla q_1(\tau)\|^2ds  \nonumber \\
& \le & \|q_1(0)\|^2 + \int^t_0\|f\|^2ds  \nonumber \\
& +&  C\int^t_0 (1+ \|h_1\|)( \|q_1(\tau)\|^2 +\|q_2(\tau)\|^2)ds
+  |M_1(\tau)|.
\end{eqnarray}
Similarly,
\begin{eqnarray} \label{q2-up}
& &  \|q_2\|^2 +  \nu  \int_0^t  \|\nabla q_2(\tau)\|^2ds  \nonumber \\
& \le & \|q_2(0)\|^2   \nonumber \\
& + &  C\int^t_0 (1+ \|h_2\|)( \|q_1(\tau)\|^2 +\|q_2(\tau)\|^2)ds
+  |M_2(\tau)|.
\end{eqnarray}
where
\begin{eqnarray*}
M_2(\tau) = 2\int^\tau_0 ({\s}_2(q_1,q_2)dW(s),q_2)ds.
\end{eqnarray*}
Adding \eqref{q1-up} and \eqref{q2-up} and taking $\sup$,
\begin{eqnarray} \label{q-sup}
& & \sup_{0\leq \tau\leq t} \|q(\tau) \|^2 +  \nu  \int_0^t \sup_{0\leq \tau\leq s} \|\nabla q (\tau)\|^2ds  \nonumber \\
& \le & \|q(0)\|^2 + \int^t_0\|f\|^2ds  \nonumber \\
& + &  C\int^t_0 (1+ \|h\|) \; \sup_{0\leq \tau\leq s} \|q
(\tau)\|^2  ds + \sup_{0\leq \tau\leq t} \; |\hat{M} (\tau)|,
\end{eqnarray}
where $\hat{M}=M_1+M_2$. By the Burkholder-Davis-Gaudy inequality,
\begin{eqnarray}
\EX\sup_{0\leq \tau\leq t}|\hat{M} (\tau)| &\leq& \EX(\int_0^t \|q
\|^2
\; \|\sigma (q)\|^2_{L_Q(\mathbb{H}_0;\mathbb{H})}ds)^{{1 \over 2}}  \nonumber\\
&\leq& \mu \EX\sup_{0\leq \tau\leq t} \|q \|^2 + C \int_0^t (1 +
\EX\sup_{0\leq \tau\leq s}\|q (\tau)\|^2 ) ds.
\end{eqnarray}

Noticing that $\|h\|^2_{L^2(0,T; {\mathbb{H}})} \leq M$, taking
$\mu$ depending on $M$ small enough such that the condition of Lemma
3.3 is satisfied and using Lemma 3.3, there exists a constant
$C_1=C_1(T, M, K, F_1, F_2 , \nu ,r)$ such that
\begin{equation} \label{E-sup}
\EX\Big(\sup_{0\leq \tau\leq t} \|q(\tau)\|^2 + \nu \int_0^t
\sup_{0\leq \tau\leq s}\|\nabla q(\tau)\|^2\, ds \Big) \leq C_1\,
\big( 1+\EX\|\xi\|^2\big), \mbox{ for all}\; 0\le t \le T.
\end{equation}

Next, we derive the estimates of $\EX\sup_{0\leq \tau\leq t}\|\nabla
q_1(\tau)\|^2$ and $\EX\sup_{0\leq \tau\leq t}\|\nabla
q_1(\tau)\|^2$. Ito formula for $\|\nabla q_2\|^2$   gives:

\begin{eqnarray*}
\|\nabla q_1(t)\|^2 &=& \|\nabla q_1(0)\|^2 + 2\nu \int_0^t(\triangle^2\psi_1,\triangle q_1)ds +
2\int_0^t(f,\triangle q_1)ds\\
&-& 2\int_0^t(( J(\psi_1,q_1+\beta y),\triangle q_1)ds  + 2 \int_0^t(\tilde{\sigma_1}(q_1,q_2)h_1,\triangle q_1)ds\\
&+& 2\int_0^t \|\nabla \sigma_1(q_1,q_2)\|^2_{L_Q(H_0; H)}ds +
2\int_0^t(\sigma_1(q_1,q_2)dW(s),\triangle q_1)ds.
\end{eqnarray*}

Since
\begin{eqnarray*}
J(\psi_1,q_1+\beta y) = J(\psi_1,q_1) + \beta {\partial \psi_1 \over
\partial x},
\end{eqnarray*}
the Cauchy-Schwartz inequality, the Young's inequality and $\|\nabla
\psi_1\|_{L^{\infty}}\le C\|\nabla \triangle \psi_1\|$ imply:
\begin{eqnarray*}
2|(J(\psi_1,q_1+\beta y),\triangle q_1)| &\leq& 2
|(J(\psi_1,q_1),\triangle q_1)| + 2\beta | ({\partial \psi_1
\over \partial x},\triangle q_1)|\\
&\leq& C_1\|\nabla \triangle \psi_1\|\;\|\nabla q_1\|\;\|\D q_1\| +
2\beta|({\partial \psi_1 \over
\partial x},\triangle q_1)|\\
&\leq& \frac{\nu}{4} \| \D q_1\|^2 + C_2\|\nabla \triangle
\psi_1\|^2\|\nabla q_1\|^2+ C_3\|{\partial \psi_1 \over \partial x}\|^2\\
&=&  C_2\|\nabla (q_1 + F_1(\nabla \psi_1 - \nabla \psi_2 )) \|^2\;\|\nabla q_1\|^2\\
& &+  C_3\|{\partial \psi_1 \over \partial x}\|^2 + \frac{\nu}{4} \|\triangle q_1\|^2\\
&\leq&  C_4(\|\nabla q_1\|^4 +  \|\nabla q_2 \|^4) + \frac{\nu}{4} \|\triangle q_1\|^2.\\
\end{eqnarray*}
Moreover, using Assumption we have:
\begin{eqnarray*}
2\|\nabla \sigma_1(q_1,q_2)\|^2_{L_Q(H_0;
H)}&=&2\|\sigma_1(q_1,q_2)\|^2_{L_Q(H_0; V)}\\
&\leq& 2K'(1 + \|\nabla q_1\|^2 + \|\nabla q_2\|^2),
\end{eqnarray*}

\begin{eqnarray*}
(f,\triangle q_1) \leq {\nu \over 4}\|\triangle q_1\|^2 + {1 \over
\nu} \|f\|^2,
\end{eqnarray*}
\begin{eqnarray*}
|(\tilde{\sigma_1}(q_1,q_2)h_1,\triangle q_1)| &\leq& {\nu \over
4}\|\triangle q_1 \|^2 + {1 \over \nu}
\|\tilde{\s}(q_1,q_2)h_1\|^2 \\
&\leq& {\nu \over 4}\|\triangle q_1 \|^2 + C(1 + \|q_1\|^2 +
\|q_2\|^2)\|h_1\|^2.
\end{eqnarray*}

Using the above estimates, Poincare inequality and Lemma
\ref{bound}, we get:
$$
\| \nabla q_1(t)\|^2 + \nu \int_0^t \|\triangle q_1(s)\|^2 ds $$
$$\leq C + \|\nabla q_1(0)\|^2 + c\int_0^t \|\nabla q\|^4ds $$
$$ +
C\int_0^t\|\nabla q(\tau)\|^2\|h_1\|^2ds + 2\int^t_0(\triangle
q_1,\sigma_1(q_1,q_2))dW(s))$$ $$\leq C + \|\nabla q_1(0)\|^2 +
c\int_0^t \|\nabla q\|^4ds$$
\begin{eqnarray}
+ C\int_0^t\|\nabla q(\tau)\|^2\|h_1\|^2ds + 2 \sup_{0\leq \tau\leq
t}|M_3(\tau)|,\label{grad1}
\end{eqnarray}
where
\begin{eqnarray*}
M_3(\tau) = \int^{\tau}_0 (\triangle q_1,\sigma_1(q_1,q_2)dW(s)).
\end{eqnarray*}

Similarly, we have:
$$
\| \nabla q_2(\tau)\|^2 + 2\nu \|\triangle q_2(\tau)\|^2 ds
$$
$$\leq C + \|\nabla q_2(0)\|^2 +
c\int_0^t \|\nabla q\|^4ds$$
\begin{eqnarray}
+ C\int_0^t\|\nabla q(\tau)\|^2\|h_1\|^2ds + 2 \sup_{0\leq \tau\leq
t}|M_4(\tau)|,\label{grad2}
\end{eqnarray}
where
\begin{eqnarray*}
M_4(\tau) = \int^{\tau}_0 (\triangle q_2,\sigma_1(q_1,q_2)dW(s)).
\end{eqnarray*}

Adding \eqref{grad1} and \eqref{grad2}, we have
$$
\| \nabla q(\tau)\|^2 + 2\nu \|\triangle q(\tau)\|^2 ds \leq C +
\|\nabla \xi\|^2 $$
\begin{eqnarray}+ c\int_0^t (\|h_1\|^2 + \|\nabla q\|^2)\|\nabla q(\tau)\|^2ds + 2 \sup_{0\leq \tau\leq
t}|\tilde{M}(\tau)|,\label{grad2}
\end{eqnarray}
where
\begin{eqnarray*}
\tilde{M}(\tau) = M_3(\tau) + M_4(\tau).
\end{eqnarray*}By the
Burkholder-Davis-Gaudy inequality and \eqref{E-sup}:
\begin{eqnarray*}
\EX\sup_{0\leq \tau\leq t}|\tilde{M}(\tau)| &\leq&
 K'\; \EX( \int_0^t \|\nabla q\|^2(1 + \|\nabla q\|^2)  ds )^{{1 \over 2}}\\
&\leq& \mu_1 \EX\sup_{0\leq \tau\leq t}\|\nabla q_1\|^2 + C \int_0^t
(1 + \EX\sup_{0\leq \tau\leq s}\|\nabla q\|^2) ds\\&\leq&\mu_1
\EX\sup_{0\leq \tau\leq t}\|\nabla q_1\|^2 + C.
\end{eqnarray*}
Taking $\mu$(depending on $M$ and the bound in \eqref{E-sup}) small
enough such that the condition of Lemma 3.3 is satisfied and using
Lemma 3.3, there exists a constant $C_2=C_2(T, M, K, K', F_1, F_2 ,
\nu ,r)$ such that

\begin{equation} \EX\sup_{0\leq \tau\leq t}\| \nabla q(\tau)\|^2 + 2\nu  \int_0^t
\EX\sup_{0\leq \tau\leq s}\|\triangle q(\tau)\|^2 ds \leq
C_2(\EX\|\nabla \xi\|^2 + 1).
\end{equation}
Thus the mild solution, is also strong solution, when $q(0) = \xi
\in$ $\mathbb{H}$$^1_0$.

\medskip

The proof of uniqueness is standard, we omit the proof here.

\end{proof}


\section{Large deviations}  \label{s4}

\subsection{Definiton}

  We consider large deviations via a    weak convergence approach
  \cite{BD00, BD07}, based on variational representations of
  infinite dimensional   Wiener processes.
  In this approach, the large deviations for SPDEs  are derived by
  showing some qualitative properties (well-posedness, compactness
  and weak convergence) of certain perturbations of the original
  SPDEs    \cite{Sundar, WangDuan, DuanMillet}.
   More information about this weak convergence approach for large deviations
   in the finite dimensional setting can be found in the book \cite{Dupuis}.
  It is different from other existing approaches,
  which usually require extra exponential tightness estimates,
  for establishing large deviation principles for
  SPDEs
  \cite{CardonWeber, Cerrai-Rockner, Chang, ChenalMillet,  Chow, FW92, Kallian-Xiong, Peszat, Sowers, Zab}.
An alternative approach \cite{Feng} for large deviations is based
on nonlinear semi-group theory and infinite dimensional
Hamilton-Jacobi equations; it also requires to establish
exponential tightness.
\bigskip

We rewrite the stochastic two-layer model to indicate its
dependence  on the small parameter $\e$:
\begin{eqnarray} \label{Benard2}
\;\;\;\;   d q^\e(t) + [Aq^\e(t) +F(q^\e(t))]dt = \sqrt{\e} \;
\s(q^\e(t)) dW(t),\;     q^\e(0)=\xi.
\end{eqnarray}
The solution is denoted as $q^\e = {\mathcal G}^\e(\sqrt{\e} W)$ for a Borel measurable function ${\mathcal G}^\e: C([0, T];
{\mathbb{H}} ) \to X.$
 We show a large deviation principle for
$q^\e$.

  The  space $X=C([0, T]; {\mathbb{H}}) \cap L^2((0, T); \mathbb{H}_0^1)$
  endowed with the metric associated with the norm defined in
\eqref{norm} is Polish. Let   $\mathcal{B}(X)$ denote its  Borel
$\s-$field.
 The theory of large deviations
\cite{BD07, FW84} is about   the exponential decay of $\PX(q^\e
\in A)$ for events $A \in {\mathcal B}(X)$ as $\e \to 0$; this
decay  is described in terms of a rate function. We recall some
definitions \cite{BD07}.
\begin{definition} (Good Rate function)\\
A function $I: X \to [0, \infty]$ is called a good rate function
on $X$ if for each $M<\infty$ the level set $\{q \in X: I(q) \leq
M \}$ is a compact subset of $X$. For $A\in \mathcal{B}(X)$, we
define $I(A)=\inf_{q \in A} I(q)$.
\end{definition}
\begin{definition} (Large deviation principle)\\
Let $I $ be   a rate function on $X$.   The random sequence
$\{q^\e \}$ is said to satisfy a large deviation principle on $X$
 with rate function $I$ if the following two conditions hold.

1. \textbf{Large deviation upper bound.} For each closed subset
$F$ of $X$:
$$
\lim\sup_{\e\to 0}\; \e \log \PX(q^\e \in F) \leq -I(F).
$$

2. \textbf{Large deviation lower bound.} For each open subset $G$
of $X$:
$$
\lim\inf_{\e\to 0}\; \e \log \PX(q^\e \in G) \geq -I(G).
$$
\end{definition}
\smallskip

The hypothesis on the growth condition and the Lipschitz property of
$\sigma$ are still the same as  {\bf (A.1) (A.2) (A.3)}.\\

The proof of the large deviation principle will use the following
technical lemma which studies time increments of the solution to
the stochastic control equation.  For any integer $k=0, \cdots,
2^n-1$, and $s\in [k T 2^{-n}, (k+1) T 2^{-n}]$, set
$\underline{s}_n= kT2^{-n}$ and $\bar{s}_n=(k+1)T 2^n$. Given
$N>0$,
 $h\in {\mathcal A}_M$,
$\e\geq 0$ small enough, let  $q_h^\e$ denote the solution to
\eqref{Benardgene} given by Theorem \ref{wellposeness}, and for
$t\in [0,T]$,  let
\[ G_N(t)=\Big\{ \omega \, :\, \Big (\|q_h(t)\|^2 \leq N \Big)\
and\ \Big(  \int_0^T \|\nabla q_h(t)\|^2\, dt  \leq N \Big)
\Big\}.\]

\begin{lemma} \label{timeincrement}
Let $M,N>0$, $\sigma$ and $\tilde{\sigma}$ satisfy the Assumptions
(A.1),(A.2) and (A.3),
 $\xi\in {\mathbb{H}}$.
 Then there exists a positive  constant
$C:=C(\nu, \kappa,K,L, T,M,N,\e_0)$ such that for any  $h\in
{\mathcal A}_M$, $\e\in [0, \e_0]$,
\begin{equation}  \label{time}
I_n(h,\e):=\EX\Big[ 1_{G_N(T)}\; \int_0^T
\|q_h^\e(s)-q_h^\e(\bar{s}_n)\|^2 \, ds\Big] \leq C\,
2^{-\frac{n}{2}}.
\end{equation}
\end{lemma}

This lemma may be similarly proved as in \cite{DuanMillet}.

We now prove weak convergence and compactness in the following two
subsections.


\subsection{Weak convergence}

Let $\e_0$ be defined as in Theorem \ref{wellposeness} and  $h_\e$
be a family of random elements taking values in ${\mathcal A}_M$.
Let $q_{h_\e}$, or more strictly speaking, $q^\e_{h_\e}$, be the
solution of the corresponding stochastic control equation  with
initial condition $q_{h_\e}(0)=\xi \in {\mathbb{H}}$:
\begin{eqnarray} \label{scontrol}
d q_{h_\e} + [Aq_{h_\e} +F(q_{h_\e})]dt
=\s(q_{h_\e}) h_\e dt+\sqrt{\e} \; \s(q_{h_\e}) dW(t) . 
\end{eqnarray}
In component form:
\begin{eqnarray*}
\frac{\partial q_{h_\e,1}}{\partial t} &+& J(\psi_{h_\e,1}, q_{h_\e,1} + \beta  y ) = \nu \D^2  \psi_{h_\e,1}\\
 &+& f   + \sigma_1(q_{h_\e,1},q_{h_\e,2})h_{\e,1} +  \sqrt{\epsilon}\sigma_1(q_{h_\e,1},q_{h_\e,2})\dot{W^1}  \; ,\\
\frac{\partial q_{h_\e,2}}{\partial t} &+& J(\psi_{h_\e,2}, q_{h_\e,2} + \beta  y )   = \nu \D^2  \psi_{h_\e,2}\\
   &-& r \D  \psi_2 + \sigma_2(q_{h_\e,1},q_{h_\e,2})h_{\e,2} + \sqrt{\epsilon}\sigma_2(q_{h_\e,1},q_{h_\e,2})\dot{W^2} \; .
\end{eqnarray*}

Note that $q_{h_\e}={\mathcal G}^\e(\sqrt{\e} W_. + \int_0^.
h_\e(s)ds )$
due to the  uniqueness of the solution. \\


For all $h \in L^2(0, T; {\mathbb{H}}_0)$, let $q_h$ be the
solution of the corresponding control equation (\ref{dcontrol})
with initial condition $q_h(0)=\xi$:
\begin{eqnarray} \label{dcontrol2}
d q_h + [Aq_h +F(q_h)]dt =\s(q_h) h dt . 
\end{eqnarray}
In component form:
\begin{eqnarray*}
\frac{\partial q_{h,1}}{\partial t} + J(\psi_{h,1}, q_{h,1} + \beta  y ) &=& \nu \D^2  \psi_{h,1}  \\
& &+ f   +  \sigma_1(q_{h,1},q_{h,2})h_1  \; ,\\
\frac{\partial q_{h,2}}{\partial t} + J(\psi_{h,2}, q_{h,2} + \beta  y ) &=& \nu \D^2  \psi_{h,2}  \\
 & &- r \D  \psi_{h,2} +  \sigma_2(q_{h,1},q_{h,2})h_2 \; .
\end{eqnarray*}

  Noting that
$\int_0^. h(s)ds \in C([0, T]; {\mathbb{H}}_0)$, we define ${\mathcal G}^0: C([0, T]; {\mathbb{H}}_0) \to C([0, T];
{\mathbb{H}}) \cap L^2((0, T); {\mathbb{H}}_0^1))$ by
$$
{\mathcal G}^0(g)=q_h   \quad \mbox{if} \quad g=\int_0^.
h(s)ds\quad \mbox{for some}\quad  h \in L^2(0, T; {\mathbb{H}}_0).
$$
If $g$ can not be represented as above, we define ${\mathcal
G}^0(g)=0$.

\begin{lemma} (Weak convergence) \label{weakconv}\\
 Suppose that $\sigma$ satisfies the Assumptions (A.1), (A.2) and (A.3).
 Let $\xi $ be ${\mathcal F}_0$-measurable such that $\EX \|\xi\|_H^2<+\infty$,
 and let $h_\e$ converge to $h$ in distribution as random elements
taking values in ${\mathcal A}_M$ (Note that here ${\mathcal A}_M$
 is endowed with the weak topology
 induced by
 the norm \eqref{norm}).
Then
 as $\e \to 0$, $q_{h_\e}$ converges in
distribution to $q_h$ in $X=C([0, T]; {\mathbb{H}}) \cap L^2((0, T); {\mathbb{H}}_0^1))$ endowed with the norm (\ref{norm}).
 That is,
  ${\mathcal G}^\e(\sqrt{\e} W_. + \int_0^. h_\e(s)ds)$  converges in
distribution to $ {\mathcal G}^0(\int_0^. h(s)ds)$ in $X$, as $\e
\to 0$.
\end{lemma}

\begin{proof}
Since ${\mathcal A}_M$ is a Polish space (complete separable
metric space), by the Skorokhod representation theorem, we can
construct processes $(\tilde{h}_\e, \tilde{h}, \tilde{W})$ such
that the joint distribution of $(\tilde{h}_\e,  \tilde{W})$ is the
same as that of $(h_\e, W)$,  the distribution of $\tilde{h}$
coincides with that of $h$, and $ \tilde{h}_\e  \to \tilde{h}$,
a.s., in the (weak) topology of $S_M$.

Let $\tilde{q}^{\e}=q_{h_\e}-q_h$, or in component form
$\tilde{q}^{\e}=(\tilde{q}_1^{\e}, \tilde{q}_2^{\e}) =
(q_1^{h_\e}-q_1^h, q_2^{h_\e}-q_2^h)$. We first derive
\begin{eqnarray}\label{difference2}
& &d \tilde{q}^{\e} + \big[A\tilde{q}^{\e} +F(q_{h_\e})-F(q_h)
\big]dt
\nonumber\\
&&\qquad  =\big[\s(q_{h_\e}) h_\e -\s(q_h) h\big] dt +\sqrt{\e} \;
\s(q_{h_\e}) dW(t), \;\; \tilde{q}^{\e}(0)=0.
\end{eqnarray}

In component form,  $\tilde{q}_1^{\e}(0) = 0,  \;
\tilde{q}_2^{\e}(0)=0$ and
\begin{eqnarray*}
\frac{\partial \tilde{q_1}^{\e}}{\partial t} + J(\psi^{h_\e}_1, q^{h_\e}_1 + \beta  y )&-&J(\psi^{h}_1, q^{h}_1 + \beta  y ) = \nu \D^2  \tilde{\psi_1}^{\e}\\
&+& \big[ (\s_1(q^{h_\e})h^\e_1 - \s_1(q^{h})h_1)\big]dt + \sqrt{\epsilon}\sigma_1(q^{h_\e})\dot{W^1} \  \; ,\\
\frac{\partial \tilde{q_2}^{\e}}{\partial t} + J(\psi^{h_\e}_2, q^{h_\e}_2 + \beta  y )&-&J(\psi^{h}_2, q^{h}_2 + \beta  y )+r\D \tilde{\psi_2}^{\e} = \nu \D^2  \tilde{\psi_2}^{\e}\\
   &+& \big[ (\s_2(q^{h_\e})h^\e_2 - \s_2(q^{h})h_2)\big]dt + \sqrt{\epsilon}\sigma_2(q^{h_\e})\dot{W^1} \  \; .\\
\end{eqnarray*}

Similar to the definition of $\|q\|_{L_2}$ in the proof of theorem \ref{wellposeness}, we define that $\|\tilde{q}^{\e}\|^2=
\|\tilde{q}_1^{\e}\|^2 + \|\tilde{q}_2^{\e}\|^2$. On any finite time interval $[0, t]$ with $t\leq T$, the It\^o's formula,
Lemmas \ref{Jacobi} and \ref{bound}, Assumption (A.2) and (A.3) yield
\begin{eqnarray}
& &\|\tilde{q}_1^{\e}(t)\|^2 +2\nu \int_0^t \|\nabla\tilde{q}^{\e}_1(s)\|^2 ds \nonumber\\
&=&  - 2 \int_0^t( J(\tilde{\psi^{\e}}_1, q_1^{h_{\e}}), \tilde{q}_1^{\e}) ds -2\beta\int_0^t( \frac{\partial \tilde{\psi_1^{\e}}}{\partial x}, \tilde{q}_1^{\e})ds \nonumber\\
& & + 2\nu F_1 \int_0^t (\D\tilde{\psi^{\e}_1} - \tilde{\psi^{\e}_2},\tilde{q}^{\e}_1) ds  + 2\int_0^t(\s_1(q_{h_\e})h^\e_1-\s_1(q_h)h_1, \tilde{q}^{\e}_1)ds \nonumber \\
& &  +  2\sqrt{\e}\int_0^t (\tilde{q}_1^{\e}, \s_1(q_{h_\e})
dW^1(s))
+ \e\int_0^t\|\s_1(q_{h_\e})Q_1^\frac12\|^2_{H.S.}ds   \nonumber\\
&\leq &   C_1  \int_0^t\|\tilde{q^{\e}}_1\| \; \|\frac{\partial \tilde{\psi^{\e}_1}}{\partial x}\|ds
 +C_2\int_0^t \| \triangle\tilde{\psi^{\e}}_1\|\; \|\nabla q^{h_{\e}}_1\| \;\|\tilde{q}^{\e}_1\|\;  ds \nonumber \\
& & + 2\nu F_1 \int_0^t (\|\D\tilde{\psi^{\e}_1}\| + \|\D\tilde{\psi^{\e}_2}\|)\; \|\tilde{q}^{\e}_1\| ds   \nonumber \\
& & + 2 \int_0^t {\|(\s_1(q_{h_\e})-\s_1(q_h))h^{\e}_1\|}\| \tilde{q}^{\e}_1\| ds  + \int_0^t(\s_1(q)(h^\e_1-h_1), \tilde{q}^{\e}_1)ds            \nonumber \\
& & + 2\sqrt{\e}\int_0^t (\tilde{q}_1^{\e}, \s_1(q_{h_\e})
dW^1(s))
+ \e K\int_0^t (1+|q_{h_\e}|^2) ds         \nonumber \\
&\leq&  {1 \over 2}C_1( \int_0^t\|\tilde{q}^{\e}_1\|^2 ds +
\int_0^t \|\frac{\partial \tilde{\psi^{\e}_1}}{\partial x}\| ^2ds)
+ {1 \over 2}C_2(
\int_0^t \| \triangle\tilde{\psi}^{\e}_1\|^2 ds + \int_0^t \|\nabla q_1^{h_\e}\|^2 \; \|\tilde{q}^{\e}_1\|^2   ds )  \nonumber \\
& & + \nu F_1 \int_0^t(\|\D\tilde{\psi^{\e}_1}\|^2 + \|\D\tilde{\psi^{\e}_2}\|^2 + 2\|\tilde{q}^{\e}_1\|^2) ds \nonumber \\
& & + 2 \int_0^t  \sqrt{L}\;\|\tilde{q^{\e}}\| \; \|h^\e\| \;
\|\tilde{q}^{\e}_1\|ds
 +\int_0^t(\s_1(q)(h^\e_1-h_1), \tilde{q}^{\e}_1)ds   \nonumber \\
& & + 2\sqrt{\e}\int_0^t (\tilde{q}_1^{\e}, \s_1(q_{h_\e})
dW^1(s))
+ \e K\int_0^t (1+|q_{h_\e}|^2) ds         \nonumber \\
&\leq&    {1 \over 2}C_1( \int_0^t\|\tilde{q}^{\e}_1\|^2 ds +
\int_0^t \|\frac{\partial \tilde{\psi^{\e}_1}}{\partial x}\| ^2ds)
+ {1 \over 2}C_2( \int_0^t \|\triangle\tilde{\psi^{\e}}_1\|^2 ds +
\int_0^t \|\nabla q^{h_
\e}_1\|^2 \; \|\tilde{q^{\e}}_1\|^2  ds ) \nonumber \\
& & + \nu F_1 \int_0^t(\|\D\tilde{\psi^{\e}_1}\|^2 + \|\D\tilde{\psi^{\e}_2}\|^2 + 2\|\tilde{q}^{\e}_1\|^2) ds \nonumber \\
& & + \int_0^t \|\tilde{q}^{\e}\|^2ds + L \int_0^t
\|h^\e\|^2\|\tilde{q}_1\|^2 ds
 + \int_0^t(\s_1(q)(h^\e_1-h_1), \tilde{q}^{\e}_1)ds  \nonumber \\ 
& & + 2\sqrt{\e}\int_0^t (\tilde{q}_1^{\e}, \s_1(q_{h_\e})
dW^1(s)) + \e K\int_0^t (1+|q_{h_\e}|^2) ds .\nonumber\\
\label{error3}
\end{eqnarray}


Similarly

\begin{eqnarray}
&& \|\tilde{q}^{\e}_2(t)\|^2 +2\nu \int_0^t \|\nabla\tilde{q}^{\e}_2(s)\|^2 ds  \nonumber \\
&=&- 2 \int_0^t( J(\tilde{\psi}^{\e}_2, q^{h_\e}_2),
\tilde{q}^{\e}_2)ds -r\int_0^t
(\bigtriangleup\tilde{\psi}^{\e}_2,\tilde{q}^{\e}_2)ds
-2\beta\int_0^t(
\frac{\partial \tilde{\psi^{\e}_2}}{\partial x}, \tilde{q}^{\e}_2)ds \nonumber\\
& & + 2\nu F_2 \int_0^t (\D(\tilde{\psi^{\e}_2} - \tilde{\psi^{\e}_1}),\tilde{q}^{\e}_2) ds + 2\int_0^t(\s_2(q_{h_\e})h^\e_2-\s_2(q_h)h_2, \tilde{q}^{\e}_2)ds \nonumber \\
&& +  2\sqrt{\e}\int_0^t (\tilde{q}^{\e}_2, \s_2(q_{h_\e})
dW^2(s))
+ \e\int_0^t\|\s_2(q_{h_\e})Q_2^\frac12\|^2_{H.S.}ds   \nonumber \\
&\leq&    C_1  \int_0^t\|\tilde{q}^{\e}_2\| \; \|\frac{\partial
\tilde{\psi^{\e}_2}}{\partial x}\|ds
 +C_2\int_0^t \| \triangle\tilde{\psi}^{\e}_2\|\; \|\nabla q^{h_\e}_2\| \; \|\tilde{q}^{\e}_2\|\;  ds  \nonumber \\
& & + C_3\int_0^t \|\bigtriangleup\tilde{\psi}^{\e}_2\| \;\|\tilde{q}^{\e}_2\|ds + 2\nu F_2 \int_0^t (\|\D\tilde{\psi^{\e}_2}\| + \|\D\tilde{\psi^{\e}_2}\|)\; \|\tilde{q}^{\e}_2\| ds   \nonumber \\
& & +  2 \int_0^t \big\{\|(\s_2(q_{h_\e})-\s_2(q_h))h^\e_2\|\big\} \|\tilde{q}^{\e}_2\|ds +  \int_0^t(\s_2(q)(h^\e_2-h_2), \tilde{q}^{\e}_2)ds            \nonumber \\
&& +  2\sqrt{\e}\int_0^t (\tilde{q}^{\e}_2, \s_2(q_{h_\e})
dW^2(s))
+ \e\int_0^t\|\s_2(q_{h_\e})Q_2^\frac12\|^2_{H.S.}ds   \nonumber \\
&\leq&  {1 \over 2}C_1( \int_0^t\|\tilde{q}^{\e}_2\|^2 ds +
\int_0^t
\|\frac{\partial \tilde{\psi^{\e}_2}}{\partial x}\| ^2ds) \nonumber \\
& &+ {1 \over 2}C_2(\int_0^t \| \triangle\tilde{\psi}^{\e}_2\|^2 ds + \int_0^t \|\nabla q^{h_\e}_2\|^2\; \|\tilde{q}^{\e}_2\|^2  ds ) \nonumber \\
& & + {1 \over 2}C_3\int_0^t( \|\bigtriangleup\tilde{\psi}^{\e}_2\|^2 + \|\tilde{q}^{\e}_2\|^2) ds + \nu F_2 \int_0^t(\|\D\tilde{\psi^{\e}_2}\|^2 + \|\D\tilde{\psi^{\e}_1}\|^2 + 2\|\tilde{q}^{\e}_2\|^2) ds \nonumber \\
& &+ 2 \int_0^t  \sqrt{L}\;\|\tilde{q}^{\e}\| \; \|h^\e\| \;
\|\tilde{q}_2\|ds
 +\int_0^t(\s_2(q)(h^\e_2-h_2), \tilde{q}^{\e}_2)ds     \nonumber \\
& &  +2\sqrt{\e}\int_0^t (\tilde{q}_2^{\e}, \s_2(q_{h_\e})
dW^2(s))
+ \e K\int_0^t(1+|q_{h_\e}|^2)ds   \nonumber \\
&\leq &   {1 \over 2}C_1( \int_0^t\|\tilde{q}^{\e}_2\|^2 ds +
 \int_0^t \|\frac{\partial \tilde{\psi^{\e}_2}}{\partial x}\| ^2ds) \nonumber\\
& &+{1 \over 2}C_2(\int_0^t \| \triangle\tilde{\psi}^{\e}_2\|^2 ds + \int_0^t \|\nabla q^{h_\e}_2\|^2\; \|\tilde{q}^{\e}_2\|^2  ds )\nonumber \\
& & + {1 \over 2}C_3 \int_0^t \|\bigtriangleup\tilde{\psi}^{\e}_2\|^2 ds  + \nu F_2 \int_0^t(\|\D\tilde{\psi^{\e}_2}\|^2 + \|\D\tilde{\psi^{\e}_1}\|^2 + 2\|\tilde{q}^{\e}_2\|^2) ds \nonumber \\
& & + \int_0^t \|\tilde{q}^{\e}\|^2ds + L \int_0^t
\|h^n\|^2\|\tilde{q}^{\e}_2\|^2 ds
 + \int_0^t(\s_2(q)(h^\e_2-h_2), \tilde{q}^{\e}_2)ds    \nonumber \\ 
& & +2\sqrt{\e}\int_0^t (\tilde{q}_2^{\e},
\s_2(q_{h_\e}) dW^2(s)) + \e K\int_0^t(1+|q_{h_\e}|^2)ds  \  . \nonumber \\
\label{error4}
\end{eqnarray}

Adding (\ref{error3}) and (\ref{error4}), we obtain an integral
inequality for $\|\tilde{q}^{\e}(t)\|^2 $
 which
involves $q_h=( q_1^h, q_2^h)$:

\begin{eqnarray} \label{total-error}
& & \|\tilde{q}(t)\|^2 +2\nu \int_0^t \|\nabla\tilde{q}(s)\|^2ds
\leq
\int_0^t \Big\{C_4 + C_5\|q_h\|^2 + L \;  \|h^{\e}\|_0^2 \Big\}\;  \|\tilde{q}^{\e}\|^2  ds \nonumber \\
&& \qquad + T_1(t,\epsilon) +  T_2(t,\epsilon)   + T_3(t,\epsilon) ,
\end{eqnarray} where
\begin{align*}
T_1(t,\e)=& 2\sqrt{\e}\int_0^t \big( \tilde{q}_\e(s), \s(q_{h_\e}(s))\,  dW(s)  \big),\\
T_2(t,\e)= & \e K  \int_0^t (1+|q_{h_\e}(s)|^2) ds , \\
T_3(t,\e)=& 2\int_0^t \big( \s(q_h(s))\, \big( h_{\e}(s)-h(
s)\big), \, \tilde{q}_\e(s)\big)\, ds.
\end{align*}

Our goal is to show that  as $\e \to 0$, $\|\tilde{q}(t)\|^2 + \int_0^t \|\nabla\tilde{q}(s)\|^2ds \to 0$ in probability, which
implies that $q_{h_\e} \to q_h$  in distribution in $C([0, T]; {\mathbb{H}}) \cap L^2((0, T); {\mathbb{H}}_0^1))$, as $\e \to
0$.

Fix $N>0$ and for $t\in [0,T]$ let
\begin{eqnarray*}  G_N(t)&=&\Big\{ \omega \, :\, \Big (\|q_h(t)\|^2 \leq N
\Big) and
\Big(  \int_0^T \|\nabla q_h(t)\|^2\, dt  \leq N \Big) \Big\} , \\
G_{N,\e}(t)&=&  G_N(t)\cap \Big\{ \|q_{h_\e}(t)\|^2 \leq N\Big\} \cap
 \Big\{ \int_0^T \|\nabla q_{h_\e}(t)\|^2\, dt\leq N
\Big\} .
\end{eqnarray*}

\emph{\textbf{Claim 1.}} For any $\e_0>0$, $ {\displaystyle
\sup_{0<\e\leq \e_0}\; \sup_{h,h_\e \in {\mathcal A}_M}
\PX(G_{N,\e}(T)^c )\to 0 \; \mbox{\rm as }\; N\to \infty.}$
\\
Indeed, for $\e>0$, $h,h_\e \in {\mathcal A}_M$, the Markov
inequality and the estimate \eqref{boundgeneral} imply
\begin{align*}
&  \PX ( G_{N,\e}(T)^c ) \leq   \PX \Big(\|q_h(t)\|^2  > N \Big)
 +\PX \Big(\|q_{h_\e}(t)\|^2  > N \Big) \\
&\qquad \qquad\qquad   +  \PX\  \Big( \int_0^T \|\nabla
q_h(t)\|^2\, dt \Big)  >N\Big)
+ \PX\Big(\int_0^T \|\nabla q_{h_\e}(t)\|^2\, dt \Big) > N \Big)\\
&\leq  \frac{1}{N}
  \sup_{ 
\;h, h_\e \in {\mathcal A}_M}
 \EX
\Big( \|q_h(t)\|^2  + \|q_{h_\e}(t)\|^2  
+ \int_0^T \|\nabla
q_h(t)\|^2\, dt +  \int_0^T \|\nabla q_{h_\e}(t)\|^2\, dt \Big) \\
&\leq
 {C_1(\nu, \k,K,L, T, M) \, \big(1+ \EX|\xi|^2\big)}{N}^{-1}.
\end{align*}
\bigskip

\emph{\textbf{Claim 2.}} For fixed $N>0$,
 $h, h_\e \in {\mathcal A}_M$  such that  as $\e \to 0$, $h_\e\to h$ a.s. in the weak topology
 of $L^2([0,T],H_0)$,
 one has  as $\e \to 0$
\begin{equation} \label{cv1}
\EX\Big[ 1_{G_{N,\e}(T)} \Big( \|\tilde{q}_h(t)\|^2 + \int_0^T
\|\nabla \tilde{q}_h(t)\|^2\, dt\Big) \Big] \to 0.
\end{equation}

The  Claim 2  can be similarly proved as in \cite{DuanMillet}.

\smallskip

To conclude the proof of the Lemma \ref{weakconv}, let $\delta>0$
and $\alpha >0$ and set
\[  \Lambda_\e := |\tilde{q}_\e|^2_X =  \|\tilde{q}_h(t)\|^2 + \int_0^T
\|\nabla \tilde{q}_h(t)\|^2\, dt.   \] Then the Markov inequality
implies that
\[
\PP(\Lambda_\e > \de )  =  \PP(G_{N,\e}(T)^c )+ \frac{1}{\delta} \EX\Big( 1_{G_{N,\e}(T)} |\tilde{q}_\e|^2_X\Big).
\]
By \emph{Claim 1}, we can choose $N$ large enough so that $
\PP(G_{N,\e}(T)^c)<\alpha$ for every $\e$. Fix $N$, {\it Claim 2}
then implies that for  $\e$ small enough,
 $ \EX\Big( 1_{G_{N,\e}(T)} |\tilde{q}_\e|^2_X\Big)
< \delta  \alpha$. This concludes the proof of the Lemma
\ref{weakconv}.
\end{proof}


\subsection{Compactness}

The following compactness result will show that the rate function
of the LDP  satisfied by the solution
 to \eqref{scontrol} is a good rate function.

\begin{lemma} (Compactness) \label{compact}\\
Let $M$ be any fixed finite positive number. Define
$$
K_M=\{q_h \in  C([0, T]; {\mathbb{H}}) \cap L^2((0, T);{\mathbb{H}}^1_0 ):  h \in S_M \},
$$
where $q_h$ is the unique solution of the control equation:
\begin{eqnarray} \label{dcontrol}
\quad d q_h(t) + \big[Aq_h(t) +F(q_h(t))\big]dt =\s(q_h(t)) h(t)
dt, \;\; q_h(0)=\xi.
\end{eqnarray}
Then $K_M$ is a compact subset in $ X$.
\end{lemma}

\begin{proof}
Let $q^n$ be a sequence in $K_M$, corresponding to solutions of
(\ref{dcontrol}) with controls $h^n$ in $S_M$:
\begin{eqnarray} \label{dcontroln}
\qquad d q^n(t) + \big[Aq^n(t) +F(q^n(t)) \big]dt =\s(q^n(t))
h^n(t) dt, \;\; q^n(0)=\xi.
\end{eqnarray}

Since $S_M$ is a bounded closed subset in the Hilbert space $L^2((0, T); {\mathbb{H}}_0)$, it is weakly compact. So there exists
a subsequence of $h^n$, still denoted as $h^n$, which converges weakly to a limit $h$ in $L^2((0, T); {\mathbb{H}}_0)$. Note
that in fact $h \in S_M$ as $S_M$ is closed. We now show that the corresponding subsequences of solutions, still denoted as
$q^n$, converges in $ X$
 to $q$ which is the solution of the
following ``limit'' equation
\begin{eqnarray} \label{limiteqn}
d q(t) + [Aq(t) +F(q(t))]dt =\s(q(t)) h(t) dt, \;\; q(0)=\xi.
\end{eqnarray}
This will complete the proof of the compactness of $K_M$.

Let $\tilde{q}=q^n-q$, or in component form
$\tilde{q}=(\tilde{q}_1, \tilde{q}_2)=(q^n_1-q_1, q^n_2-q_2)$.
\begin{eqnarray}\label{difference}
\qquad d \tilde{q} + [A\tilde{q} +F(q^n)-F(q)]dt =[\s(q^n) h^n
-\s(q) h] dt, \;\; \tilde{q}(0)=0.
\end{eqnarray}
In component form,  $\tilde{q}_1(0)=0$, $\tilde{q}_2(0)=0$ and
\begin{eqnarray*}
\frac{\partial \tilde{q_1}}{\partial t} + J(\psi^n_1, q^n_1 + \beta  y )&-&J(\psi_1, q_1 + \beta  y ) = \nu \D^2  \tilde{\psi_1} \\
         &+& \big[ (\s_1(q^n)-\s_1(q))h^n_1+\s_1(q)(h^n_1-h_1)\big]\  \; ,\\
\frac{\partial \tilde{q_2}}{\partial t} + J(\psi^n_2, q^n_2 + \beta  y )&-& J(\psi_2, q_2 + \beta  y )= \nu \D^2  \tilde{\psi_2 } - r \D  \tilde{\psi_2 } \\
         &+& \big[(\s_2(q^n)- \s_2(q)) h^n_2+\s_2(q)(h^n_2- h_2)\big]\ \; .
\end{eqnarray*}
After the following transformations:
\begin{eqnarray*}
& &\frac{\partial \tilde{q_1}}{\partial t} + J(\psi^n_1, q^n_1 + \beta  y ) - J(\psi_1, q_1 + \beta  y )\\
&=&  \nu \D^2  \tilde{\psi_1}  + \big[ (\s_1(q^n)-\s_1(q))h^n_1+\s_1(q)(h^n_1-h_1)\big]\  \; .\\
\end{eqnarray*}
Since the Jacobian operator is bilinear:
\begin{eqnarray*}
& & \frac{\partial \tilde{q_1}}{\partial t} + J(\psi^n_1, q^n_1) -
J(\psi_1, q_1 ) + \beta\frac{\partial \tilde{\psi_1}}{\partial
x}\\
&=& \nu \D^2  \tilde{\psi_1}  + \big[ (\s_1(q^n)-\s_1(q))h^n_1+\s_1(q)(h^n_1-h_1)\big]\  \; .\\
\end{eqnarray*}
By adding and subtracting the same item $J(\psi^n_1, q_1 )$, we
get:
\begin{eqnarray*}
& &\frac{\partial \tilde{q}_1}{\partial t} + J(\psi^n_1, q^n_1 ) -
J(\psi^n_1, q_1 ) + J(\psi^n_1, q_1 )  - J(\psi_1, q_1  ) +
\beta\frac{\partial \tilde{\psi_1}}{\partial x}\\
&=& \nu \D^2  \tilde{\psi_1}  + \big[ (\s_1(q^n)-\s_1(q))h^n_1+\s_1(q)(h^n_1-h_1)\big]\  \; .\\
\end{eqnarray*}
Finally, we have:
\begin{eqnarray*}
& &\frac{\partial \tilde{q_1}}{\partial t} + J(\psi^n_1,
\tilde{q}_1 ) + J(\tilde{\psi}_1, q_1 ) +
\beta\frac{\partial \tilde{\psi_1}}{\partial x}\\
&=& \nu \D(\tilde{q_1} +  F_1(\tilde{\psi_1} - \tilde{\psi_2}))  + \big[ (\s_1(q^n)-\s_1(q))h^n_1+\s_1(q)(h^n_1-h_1)\big]\  \; .\\
\end{eqnarray*}

Thus, on any finite time interval $[0, T]$, the It\^o's formula, (A.3) and the Young's inequality imply
\begin{eqnarray*}
&& \|\tilde{q}_1(t)\|^2 + 0 + 2 \int_0^t( J(\tilde{\psi}_1, q_1),
\tilde{q}_1) ds + 2\beta\int_0^t( \frac{\partial
\tilde{\psi_1}}{\partial x},\ \tilde{q}_1)ds  \\
\nonumber \\
&=&  2\nu \int_0^t (\D\tilde{q}_1(s),\tilde{q}_1(s))ds + 2\nu F_1 \int_0^t (\D(\tilde{\psi_1} - \tilde{\psi_2}),\tilde{q}_1) ds\\
&& + 2\int_0^t(\s_1(q^n)h^n_1-\s_1(q)h_1,\ \tilde{q}_1)ds. \nonumber \\
\nonumber \\
\end{eqnarray*}

\begin{eqnarray}
&& \|\tilde{q}_1(t)\|^2 +2\nu \int_0^t \|\nabla\tilde{q}_1(s)\|^2 ds \nonumber\\
&=&   - 2 \int_0^t( J(\tilde{\psi}_1, q_1), \tilde{q}_1) ds - 2\beta\int_0^t( \frac{\partial \tilde{\psi_1}}{\partial x}, \tilde{q}_1)ds \nonumber \\
&& + 2\nu F_1 \int_0^t (\D(\tilde{\psi_1} - \tilde{\psi_2}),\tilde{q}_1) ds + 2\int_0^t(\s_1(q^n)h^n_1-\s_1(q)h_1, \tilde{q}_1)ds \nonumber \\
&\leq&    C_1  \int_0^t\|\tilde{q}_1\| \; \|\frac{\partial
\tilde{\psi_1}}{\partial x}\|ds +C_2\int_0^t \| \triangle\tilde{\psi}_1\|\;\|\nabla q_1\| \; \|\tilde{q}_1\|\;  ds  \nonumber \\
&&+  2\nu F_1 \int_0^t (\|\D\tilde{\psi_1}\| + \|\D\tilde{\psi_2}\|)\; \|\tilde{q}_1\| ds   \nonumber \\
&&+  2 \int_0^t \big\{\|(\s_1(q^n)-\s_1(q))h^n_1\|\big\} \|\tilde{q}_1\|ds + 2\int^t_0(\s_1(q)(h^n_1-h_1),\tilde{q}_1) ds              \nonumber \\
&\leq&  {1 \over 2}C_1( \int_0^t\|\tilde{q}_1\|^2 ds + \int_0^t
\|\frac{\partial \tilde{\psi_1}}{\partial x}\| ^2ds) + {1 \over
2}C_2(
\int_0^t \| \triangle\tilde{\psi}_1\|^2 ds + \int_0^t \|\nabla q_1\|^2\; \|\tilde{q}_1\|^2   ds ) \nonumber \\
&& + \nu F_1 (\int_0^t(\|\D\tilde{\psi_1}\|^2 + \|\D\tilde{\psi_2}\|^2) + 2\|\tilde{q}_1\|^2 ds) \nonumber \\
&& + 2 \int_0^t  \sqrt{L}\;\|\tilde{q}\| \; \|h^n\| \;
\|\tilde{q}_1\|ds
 2\int^t_0(\s_1(q)(h^n_1-h_1),\tilde{q}_1) ds   \nonumber \\
&\leq&   {1 \over 2}C_1( \int_0^t\|\tilde{q}_1\|^2 ds + \int_0^t
\|\frac{\partial \tilde{\psi_1}}{\partial x}\| ^2ds) + {1 \over
2}C_2(
\int_0^t \| \triangle\tilde{\psi}_1\|^2 ds + \int_0^t \|\nabla q_1\|^2\; \|\tilde{q}_1\|^2  ds ) \nonumber \\
\nonumber \\
&& + \nu F_1 (\int_0^t(\|\D\tilde{\psi_1}\|^2 + \|\D\tilde{\psi_2}\|^2) + 2\|\tilde{q}_1\|^2 ds) \nonumber \\
&& + \int_0^t \|\tilde{q}\|^2ds + L \int_0^t
\|h^n\|^2\|\tilde{q}_1\|^2 ds
 +2\int^t_0(\s_1(q)(h^n_1-h_1),\tilde{q}_1) ds . \nonumber\\
\label{error1}
\end{eqnarray}

Similarly,

\begin{eqnarray*}
&& \|\tilde{q}_2(t)\|^2 + 0 + 2 \int_0^t( J(\tilde{\psi}_2, q_2),
\tilde{q}_2) ds + 2\beta\int_0^t( \frac{\partial
\tilde{\psi_2}}{\partial x}, \tilde{q}_2)ds  \\
\nonumber \\
&=& -r\int_0^t (\bigtriangleup\psi_2,\tilde{q}_2)ds + 2\nu \int_0^t (\D\tilde{q}_2(s),\tilde{q}_2(s))ds \\
&& + 2\nu F_2 \int_0^t (\D(\tilde{\psi_2} - \tilde{\psi_1}),\tilde{q}_2) ds + 2\int_0^t(\s_2(q^n)h^n_2-\s_2(q)h_2, \tilde{q}_2)ds.\nonumber \\
\end{eqnarray*}
\begin{eqnarray}
&& \|\tilde{q}_2(t)\|^2 +2\nu \int_0^t \|\nabla\tilde{q}_2(s)\|^2 ds \nonumber\\
&=& - 2 \int_0^t( J(\tilde{\psi}_2, q_2), \tilde{q}_2)ds -r\int_0^t (\bigtriangleup\psi_2,\tilde{q}_2ds -2\beta\int_0^t(
\frac{\partial \tilde{\psi_2}}{\partial x}, \tilde{q}_2)ds \nonumber \\
\nonumber \\
& & + 2\nu F_2 \int_0^t (\D(\tilde{\psi_2} - \tilde{\psi_1}),\tilde{q}_2) ds + 2\int_0^t(\s_2(q^n)h^n_2-\s_2(q)h_2, \tilde{q}_2)ds \nonumber \\
&\leq&   C_1  \int_0^t\|\tilde{q}_2\| \; \|\frac{\partial \tilde{\psi_2}}{\partial x}\|ds
 +C_2\int_0^t \| \triangle\tilde{\psi}_2\|\; \|\nabla q_2\| \; \|\tilde{q}_2\|\;  ds  \nonumber \\
&& + C_3\int_0^t \|\bigtriangleup\psi_2\| \;\|\tilde{q}_2\|ds + 2\nu F_2 \int_0^t (\|\D\tilde{\psi_2}\| + \|\D\tilde{\psi_2}\|)\; \|\tilde{q}_2\| ds   \nonumber \\
&& +  2 \int_0^t \big\{\|(\s_2(q^n)-\s_2(q))h^n_2\|\big\} \|\tilde{q}_2\| ds   + 2\int^t_0(\s_2(q)(h^n_2-h_2),\tilde{q}_2) ds             \nonumber \\
&\leq&  {1 \over 2}C_1( \int_0^t\|\tilde{q}_2\|^2 ds + \int_0^t
\|\frac{\partial \tilde{\psi_2}}{\partial x}\| ^2ds) + {1 \over
2}C_2(
\int_0^t \| \triangle\tilde{\psi}_2\|^2 ds + \int_0^t \|\nabla q_2\|^2\; \|\tilde{q}_2\|^2  ds )\nonumber \\
&& + {1 \over 2}C_3\int_0^t (\|\bigtriangleup\psi_2\|^2 + \|\tilde{q}_2\|^2) ds + \nu F_2 (\int_0^t(\|\D\tilde{\psi_2}\|^2 + \|\D\tilde{\psi_1}\|^2 + 2\|\tilde{q}_2\|^2) ds) \nonumber \\
&& + 2 \int_0^t  \sqrt{L}\;\|\tilde{q}\| \; \|h^n\| \;
\|\tilde{q}_2\|ds+
 2\int^t_0(\s_2(q)(h^n_2-h_2),\tilde{q}_2) ds    \nonumber \\
&\leq&   {1 \over 2}C_1( \int_0^t\|\tilde{q}_2\|^2 ds + \int_0^t
\|\frac{\partial \tilde{\psi_2}}{\partial x}\| ^2ds) + {1 \over
2}C_2(
\int_0^t \| \triangle\tilde{\psi}_2\|^2 ds + \int_0^t \|\nabla q_2\|^2\; \|\tilde{q}_2\|^2  ds ) \nonumber \\
&&+ {1 \over 2}C_3 \int_0^t \|\bigtriangleup\psi_2\|^2 ds  + \nu F_2 (\int_0^t(\|\D\tilde{\psi_2}\|^2 + \|\D\tilde{\psi_1}\|^2 + 2\|\tilde{q}_2\|^2) ds) \nonumber \\
&& + \int_0^t \|\tilde{q}\|^2ds + L \int_0^t
\|h^n\|^2\|\tilde{q}_2\|^2 ds
 + 2\int^t_0(\s_2(q)(h^n_2-h_2),\tilde{q}_2) ds  . \nonumber\\
\label{error2}
\end{eqnarray}

Adding (\ref{error1}) and (\ref{error2}), and by the Theorem
(\ref{wellposeness}) we obtain an integral inequality for
$\|\tilde{q}(t)\|^2= \|\tilde{q}_1(t)\|^2+\|\tilde{q}_2(t)\|^2$
which involves $q=\{q_1, q_2\}$:
\begin{align} \label{error}
 \|\tilde{q}(t)\|^2 +2\nu & \int_0^t \|\nabla\tilde{q}(s)\|^2ds
  \leq 2\int^t_0(\s(q)(h^n-h),\tilde{q}) ds    \nonumber \\
 & +  \int_0^t \Big\{C_4 + C_5\|q\|^2
 +   L \;  |h^n|^2 \Big\}
\; \|\tilde{q}\|^2 ds.
\end{align}

As in \cite{DuanMillet}, for $N\geq 1 $  and $k=0,\cdots,2^N$, set $t_k=k2^{-N}$. For $s\in [t_{k-1}, t_k]$, $1\leq k\leq 2^N$,
let $\bar{s}_N=t_k$. The inequality \eqref{control-norm2} implies that there exists a constant $\bar{C}>0$ such that
 \[ \sup_n \Big[ \|q(t)\|^2 + \int_0^T
\|\nabla \tilde{q}(t)\|^2\, dt + \|\tilde{q}_n(t)\|^2 + \int_0^T
\|\nabla \tilde{q}_n(t)\|^2\, dt \Big] =  \bar{C} <+\infty.\]
 Thus the Gronwall's inequality implies
\begin{equation} \label{errorbound}
 \sup_{ t\leq T} |\tilde{q}_n(t)|^2_X  \leq
\exp\Big(\int_0^t \Big\{C_4 + C_5\|q\|^2
 +   L \;  |h^n|^2 \Big\}
\; \|\tilde{q}\|^2 
 ds\Big)  \sum_{i=1}^4 I_{n,N}^{i}  ,
\end{equation}
where
\begin{eqnarray*}
I_{n,N}^1&=& \int_0^T \big| \big( \s(q(s))\,  [h_n(s)- h(s)]\, ,\, \tilde{q}_n(s)-\tilde{q}_n(\bar{s}_N)\big)\big|\, ds,  \\
I_{n,N}^2&=& \int_0^T\Big|  \Big( \big[ \s(q(s)) -
\s(q(\bar{s}_N))\big]  [h_n(s)-h(s)]\, ,\,
\tilde{q}_n(\bar{s}_N)\Big)\Big| \, ds,  \\
I_{n,N}^3&=& \sup_{1\leq k\leq 2^N} \sup_{t_{k-1}\leq t\leq t_k}
\Big|\Big( \s(q(t_k ))
 \int_{t_{k-1 }}^t (h_\e(s)-h(s))
ds \; ,\; \tilde{q}_\e(t_k) \Big)\Big| , \\
I_{n,N}^4&=& \sum_{k=1}^{2^N} \Big( \s(q(t_k )) \,
\int_{t_{k-1}}^{t_k } [ h_n(s)-h(s)]\, ds   \; ,\; \tilde{q}_n(t_k
) \Big) .
\end{eqnarray*}
The  Cauchy-Schwarz  inequality, (A.2), (A.3) and Lemma
\ref{timeincrement} imply that for some constant $C$ which does
not depend on $n$,
\begin{align} \label{estim1}
I_{n,N}^1 &\leq  \Big( \int_0^T\!\! K (1+\bar{C})
\|h_n(s)-h(s)\|_0^2 ds \Big)^{\frac{1}{2}} \Big( 2 \int_0^T \!\!
\big( \|q_n(s)-q_n(\bar{s}_N)\|^2 + \|q(s)-q(\bar{s}_N)\|^2\big)
ds
 \Big)^{\frac{1}{2}}
\nonumber\\
&\leq C 2^{-\frac{N}{4}} \, ,\\
I_{n,N}^2&\leq  \Big( L\int_0^T \|q(s)-q(\bar{s}_N)\|^2 ds
\Big)^{\frac{1}{2}}
 \Big( \bar{C} \int_0^T \|h_n(s)-h(s)\|_0^2\, ds\Big)^{\frac{1}{2}} \leq C 2^{-\frac{N}{4}}\, ,
\label{estim2}\\
I_{n,N}^3&\leq K \big( 1+\|q(t)\|) \big( \|q(t)\|+\|q_n(t)\|\big)
2^{-\frac{N}{2}} 2M \leq C \tilde{C} 2^{-\frac{N}{2}}\,  .
\label{estim3}
\end{align}
We now use a time discretization argument from Proposition 4.4 of
\cite{DuanMillet}. For given $\alpha
>0$, one may choose $N$ large enough to have $ \sup_n \sum_{i=1}^3
I_{n,N}^i \leq \alpha$. Then, for fixed $N$ and $k=1, \cdots,
2^N$,
 as $n\to \infty$, the weak convergence of $h_n$ to $h$ implies
that of $\int_{t_{k-1}}^{t_k} (h_n(s)-h(s))ds$ to 0 weakly in
$\mathbb{H}_0$. Since $\s(q(t_k))$ is a compact operator, we
deduce that for fixed $k$ the sequence $\s(q(t_k))
\int_{t_{k-1}}^{t_k} (h_n(s)-h(s))ds$ converges to 0 strongly in
$\mathbb{H}$ as $n\to \infty$. Since $\sup_n \sup_k
\|\tilde{q}_n(t_k)\|\leq 2\tilde{C}$, we have
 $\lim_n I_{n,N}^3=0$. Thus as $ n \to \infty$,
$
\|\tilde{q}_n(t)\|^2 \to 0. 
$ Using this convergence and \eqref{error}, we deduce that
$|\tilde{q}|_X \to 0$ as $n\to \infty$.
This shows that every sequence in $K_M$ has a convergent
subsequence. Hence $K_M$ is  a compact subset of $X$.
\end{proof}

With the above results, we have the following theorem.
\begin{theorem} (Large deviation principle) \label{LDP} \\
Let $q^\e$ be the solution of the stochastic two-layer problem
\begin{eqnarray}
\ \ \ \ \ \ \ \ \ d q^\e + [Aq^\e +F(q^\e)]dt = \sqrt{\e} \; \s(q^\e) dW(t),\;\;    q^\e(0)=\xi \in {\mathbb{H}}.
\end{eqnarray}
Then $\{q^\e\}$ satisfies the large deviation principle, in $C([0, T]; {\mathbb{H}}) \cap L^2((0, T); {\mathbb{H}}^1_0)$ with
the good rate function
\begin{eqnarray}
 I_\xi (\psi)= \inf_{\{h \in L^2(0, T; {\mathbb{H}}_0): \; \psi =G^0(\int_0^. h(s)ds) \}}
 \Big\{\frac12 \int_0^T \|h(s)\|_0^2 ds \Big\}.
\end{eqnarray}
Here the infimum of an empty set is taken as infinity.
\end{theorem}

\begin{proof}
Lemma \ref{compact} and Lemma \ref{weakconv} imply that $\{q^\e\}$
satisfies the Laplace principle which is equivalent to the large
deviation principle in $X=C([0, T]; {\mathbb{H}}) \cap L^2((0, T);
{\mathbb{H}}^1_0))$ with the above-mentioned rate function;
 see Theorem 4.4 in \cite{BD00} or
Theorem 4.4 in \cite{BD07}.
\end{proof}


\bigskip

{\bf Acknowledgement.}\\
Jinqiao Duan would like to thank Annie Millet for very helpful
discussions. This work was partly supported by the NSF Grant
0620539, the Cheung Kong Scholars Program and the K. C. Wong
Education Foundation,  NSF of China (No.10571087, No. 10871097),
SRFDP No. 20050319001, a Jiangsu Province NSF Grant BK2006523 and
the Teaching and Research Award Program for Outstanding Young
Teachers in Nanjing Normal University (2005--2008).


\end{document}